\documentclass{amsart}
\usepackage{amsthm}
\usepackage{amsmath}
\usepackage{amsfonts}
\usepackage{amssymb}
\usepackage[mathscr]{euscript}
\usepackage[hidelinks]{hyperref} 
\usepackage{color}
\usepackage{pstricks}
\usepackage{pst-node}
\usepackage{pstricks-add}
\numberwithin{equation}{section}
\theoremstyle{plain} 
\newtheorem{lem}[equation]{Lemma}

\newtheorem{thm}[equation]{Theorem}

\theoremstyle{definition}

\theoremstyle{remark}

\newtheorem{rem}[equation]{Remark}

\def\mpush#1.{{#1}_{\sharp}} 
\def\edges#1.{{\rm\normalfont 
    Edge}({\setbox0=\hbox{#1\unskip}\ifdim\wd0=0pt n \else #1\fi})}
\def\vertices#1.{{\rm\normalfont 
    Vertex}({\setbox0=\hbox{#1\unskip}\ifdim\wd0=0pt n \else #1\fi})}
\def\hmeas#1.{\mathscr{H}^{\setbox0=\hbox{$#1\unskip$}\ifdim\wd0=0pt 1
    \else #1\fi}} 
\def\mrest{{\,\evansgariepyrest}}
\newcommand{\evansgariepyrest}{\mathbin{\vrule height 1.3ex depth%
    0pt width 0.08ex\vrule height 0.08ex depth 0pt width 1.0ex}}
\def\natural{{\mathbb N}} 
\def\cantor{{\mathcal C}} 
\DeclareMathOperator\reg{Reg} 
\def\real{{\mathbb{R}}}
\begin{document}
\title{Poincar\'e inequalities for mutually singular measures}
\author{Andrea Schioppa}
\address{NYU/Courant Institute}
\email{schioppa@cims.nyu.edu}
\begin{abstract}
  Using an inverse system of metric graphs as in
  \cite{cheeger_inverse_poinc}, we provide a simple example of a metric space $X$ that
  admits Poincar\'e inequalities for a continuum of mutually singular measures.
\end{abstract}
\maketitle
\tableofcontents
\newcount\outline
\outline=0
\section{Introduction}
\label{sec:intro}

\ifnum\outline>0{\textbf{Outline}
\begin{enumerate}
\item Motivation for the problem: Poincar\'e inequality in $\real^n$
  holds for measures absolutely continuous wrt.~Lebesgue measure;
\item One might ask if this holds also in metric measure space.
\end{enumerate}}\fi
In this note we provide a simple example of a metric measure space
$X$ that satisfies abstract Poincar\'e inequalities in the sense of
Heinonen-Koskela \cite{heinonen98} for a $1$-parameter family of mutually singular
measures. By a recent announcement of M.~Cs\"ornyei and P.~Jones the
classical Rademacher's Theorem is \emph{sharp} in the sense that if
its \emph{conclusion} holds for the metric measure space
$(\real^n,\mu)$, then $\mu$ must be absolutely continuous with respect
to the Lebesgue measure.  From this, using the theory of
\emph{differentiability spaces} (compare \cite[Sec.~14]{cheeger99}), it follows that if $\mu$ is a doubling
measure on $\real^n$ such that the metric measure space
$(\real^n,\mu)$ admits an abstract Poincar\'e inequality, then the
measure $\mu$ is absolutely continuous with respect to the Lebesgue
measure. The example presented here shows that for metric measure
spaces a similar phenomenon does not hold; in particular, the measure
class for which Cheeger's generalization of Rademacher's Theorem holds is not
uniquely determined.
\begin{thm}
  \label{thm:exa}
  There is a compact geodesic metric space $X$ and there is a
  family of doubling probability measures $\{\mu_w\}_{w\in(0,1)}$ defined on
  $X$ such that:
  \begin{itemize}
  \item Each metric measure space $(X,\mu_w)$ supports a
    $(1,1)$-Poincar\'e inequality;
  \item If $w\ne w'$ the measures $\mu_w$ and $\mu_{w'}$ are mutually singular.
  \end{itemize}
\end{thm}
\subsection*{Acknowledgements} The author thanks B.~Kleiner for
encouraging him to write this note.
\section{The Example}
\label{sec:exa}

The goal of this Section is to prove Theorem
\ref{thm:exa}. The metric space $X$ that we consider is the example
\cite[pg.~290]{lang_plaut}, compare also
\cite[Exa.~1.2]{cheeger_inverse_l1}. We briefly recall the
construction. We build a sequence of graphs $\{X_n\}_{n\ge0}$ starting
with $X_0$, which consists of a single edge which we identify with the interval
$[0,1]$.  The graph $X_{n+1}$ is obtained from $X_n$ by subdividing
each edge of $X_n$ into four equal parts and replacing it by a
rescaled copy of the diamond graph (Figure \ref{fig:ldiamond}, where
we have labelled the edges for future reference).
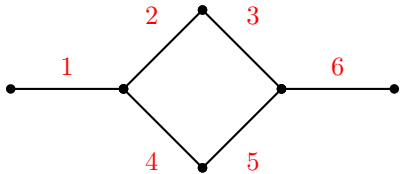
\begin{figure}[h]
  \centering
  \begin{pspicture}(4,4)
  \psset{unit=1.5}
  \def\ldiamondgraph{ 
    \psline[arrows=*-*](0,2)(1,2)
    \psline[arrows=*-*](1,2)(1.70,2.70)
    \psline[arrows=*-*](1.70,2.70)(2.40,2)
    \psline[arrows=*-*](2.40,2)(3.40,2)
    \psline[arrows=*-*](1,2)(1.70,1.30)
    \psline[arrows=*-*](1.70,1.30)(2.40,2)
    \rput(0.5,2.2){\textcolor{red}{1}}
    \rput(1.25,2.65){\textcolor{red}{2}}
    \rput(2.15,2.65){\textcolor{red}{3}}
    \rput(1.25,1.35){\textcolor{red}{4}}
    \rput(2.15,1.35){\textcolor{red}{5}}
    \rput(2.90,2.2){\textcolor{red}{6}}
  }
  \rput(0,0){\ldiamondgraph}
\end{pspicture}
\caption{Diamong graph with labelled edges}
  \label{fig:ldiamond}
\end{figure}
Each edge of $X_n$ has length $4^{-n}$ and the map which collapses the
diamond graph to a segment gives rise to $1$-Lipschitz maps
$\pi_{n+1,n}:X_{n+1}\to X_n$. In this way, the graphs
$\{X_n\}_{n\ge0}$ fit into an inverse system and, for each pair $(k,n)$
of nonnegative integers satisfying $k\le n$, one has a $1$-Lipschitz
map $\pi_{n,k}:X_n\to X_k$. The maps $\{\pi_{n,k}\}_{n\ge0,k\ge0}$
satisfy the compatibility relations:
\begin{equation}
  \label{eq:fin_compa}
  \pi_{k,m}\circ\pi_{n,k}=\pi_{n,m}.
\end{equation}
Having equipped the graphs $X_n$ with the length distance, the space
$X$ is the Gromov-Hausdorff limit of the sequence $\{X_n\}_{n\ge0}$; by the
properties of the inverse system one also gets $1$-Lipschitz maps
$\pi_{\infty,n}:X\to X_n$ which satisfy the compatibility conditions:
\begin{equation}
  \label{eq:inf_compa}
  \pi_{n,k}\circ\pi_{\infty,n}=\pi_{\infty,k}.
\end{equation}
We will let $\edges .$ and $\vertices .$ denote, respectively, the
sets of edges and vertices of $X_n$.
\par We now turn to the construction of the family of measures
$\{\mu_w\}_{w\in(0,1)}$. For $w\in(0,1)$, the measure $\mu_w$ is the
unique probability measure on $X$ which satisfies the following
requirements:
\begin{description}
\item[(R1)] For each nonnegative integer $n$, one has
  $\mpush\pi_{\infty,n}.(\mu_w)=\mu_{w,n}$ where $\mu_{w,n}$ is a
  probability measure on $X_n$;
\item[(R2)] The measure $\mu_{w,n}$ is a multiple of arclength on each
  edge of $X_n$ and $\mu_0$ is identified with the Lebesgue measure on
  $[0,1]$;
\item[(R3)] For each edge $e_n\in\edges .$, denoting by
  $\{e_{n+1,i}\}_{i=1,\cdots,6}\subset\edges n+1.$ the edges of $X_{n+1}$ whose union is
  $\pi^{-1}_{n+1,n}(e_n)$ (labelling as in Figure \ref{fig:ldiamond}), one has:
  \begin{equation}
    \label{eq:measure_diff}
    \mu_{w,n+1}\mrest e_{n+1,i} =
    \begin{cases}
      \mu_{w,n}(e_n)\,\hmeas . \mrest e_{n+1,i} & \text{if $i=1,6$}\\
      w\mu_{w,n}(e_n)\,\hmeas . \mrest e_{n+1,i} & \text{if $i=2,3$}\\
(1-w)\mu_{w,n}(e_n)\,\hmeas . \mrest e_{n+1,i} & \text{if $i=4,5$},
    \end{cases}
  \end{equation}
  where $\hmeas .$ denotes the $1$-dimensional Hausdorff measure.
\end{description}
By the main result of \cite[Thm.~1.1]{cheeger_inverse_poinc} each metric measure
space $(X,\mu_w)$ admits a $(1,1)$-Poincar\'e inequality. Note that
the class of spaces considered in \cite{cheeger_inverse_poinc} is much
broader and in this example one can also prove the $(1,1)$-Poincar\'e
inequality directly by using pencils of curves similarly as in \cite{laakso_poinc}.
\par We now turn to a probabilistic description of the points in
$X$. Let $V\subset X$ be the set of points which project to some vertex:
\begin{equation}
  \label{eq:set_v}
  V=\left\{p\in X:\exists n\ge0:\pi_{\infty,n}(p)\in\vertices .\right\};
\end{equation}
then $\mu_w(V)=0$ by conditions \textbf{(R1)} and \textbf{(R2)}.
Let $\cantor$ denote the Cantor set $\{1,\cdots,6\}^\natural$ and let
$A:\cantor\to X\setminus V$ be the map defined as follows: given $\tilde
p\in\cantor$, $A(\tilde p)$ is the unique point $p\in X\setminus V$
such that, for each $n\ge1$, $\pi_{n+1}(p)$ belongs to the edge
labelled by the integer $\tilde p(n)$ among those in
$\pi^{-1}_{n+1,n}(e_n)$ (see Figure~\ref{fig:ldiamond}), where $e_n$ is the unique edge of $X_n$
containing $\pi_n(p)$. We now define the probability measure $\nu_w$
on $\{1,\cdots,6\}$ defined as follows:
\begin{equation}
  \label{eq:prob_1_6}
  \nu_w(i)=
  \begin{cases}
    \frac{1}{4} & \text{if $i=1,6$}\\
    \frac{w}{4} & \text{if $i=2,3$}\\
    \frac{1-w}{4} & \text{if $i=4,5$};
  \end{cases}
\end{equation}
we then let $P_w$ denote the probability measure on $\cantor$ which is
the product of countably many copies of the measure $\nu_w$. We
observe that $\mpush A.P_w=\mu_w$ and we let $T_{i,n}$ denote the
random variable:
\begin{equation}
  \label{eq:random_var}
  T_{i,n}(\tilde p)=
  \begin{cases}
    1 & \text{if $\tilde p(n)=i$}\\
    0 & \text{otherwise.}
  \end{cases}
\end{equation}
\begin{lem}
  \label{lem:strong_law}
  Let $S_{i,n}=\sum_{k\le n}T_{i,k}$; then one has $P_w$-a.s.:
  \begin{equation}
    \label{eq:strong_law_s1}
    \lim_{n\to\infty}\frac{S_{i,n}}{n}=
    \begin{cases}
          \frac{1}{4} & \text{if $i=1,6$}\\
    \frac{w}{4} & \text{if $i=2,3$}\\
    \frac{1-w}{4} & \text{if $i=4,5$}.
    \end{cases}
  \end{equation}
\end{lem}
\begin{proof}
  This follows from the Strong Law of Large Numbers as the  random
  variables $\{T_{i,n}\}_{n\ge1}$ are i.i.d.
\end{proof}
We now complete the proof of Theorem~\ref{thm:exa}:
\begin{lem}
  \label{lem:singularity}
  If $w\ne w'$ the measures $\mu_w$ and $\mu_{w'}$ are mutually singular.
\end{lem}
\begin{proof}
  We show that the Radon-Nikodym derivative $\frac{d\mu_w}{d\mu_{w'}}$
  is null. We let $\reg(w')$ denote the set of points of $\cantor$ such that
  (\ref{eq:strong_law_s1}) holds (for the parameter $w'$). By an
  application of measure differentiation, if
  we consider $p\in X\setminus V$ and let $e_n(p)\in\edges .$ denote the unique edge of $X_n$
  containing $\pi_n(p)$, then for $\mu_{w'}$-a.e.~$p\in
  A(\reg(w'))$ we have:
  \begin{equation}
    \label{eq:singularity_p1}
    \begin{split}
    \frac{d\mu_w}{d\mu_{w'}}(p)&=\lim_{n\to\infty}\frac{\mu_{w,n}(e_n(p))}{\mu_{w',n}(e_n(p))}\\
    &=
    \lim_{n\to\infty}\underbrace{\left(\frac{w}{w'}\right)^{S_{2,n}(A^{-1}(p))+S_{3,n}(A^{-1}(p))}
    \left(\frac{1-w}{1-w'}\right)^{S_{4,n}(A^{-1}(p))+S_{5,n}(A^{-1}(p))}}_{c_n};
  \end{split}
\end{equation}
note that here we used that the set $\pi^{-1}_n(e_n(p))$ is comparable to the
ball of radius $4^{-n}$ about $p$. We now prove that
$\lim_{n\to\infty}c_n=0$ by showing that $\frac{\ln c_n}{n}$ converges to a
negative number. Applying (\ref{eq:strong_law_s1}) we find:
\begin{equation}
  \label{eq:singularity_p2}
  \lim_{n\to\infty}\frac{\ln c_n}{n}=\frac{1}{2}\left(w'\ln\frac{w}{w'}+(1-w')\ln\frac{1-w}{1-w'}\right);
\end{equation}
as the function $\ln$ is strictly concave and as $w\ne w'$,
\begin{equation}
  \label{eq:singularity_p3}
  \lim_{n\to\infty}\frac{\ln c_n}{n}<\frac{1}{2}\ln\left(w'\frac{w}{w'}+(1-w')\frac{1-w}{1-w'}\right)=0.
\end{equation}
\end{proof}
\begin{rem}
  \label{rem:diff}
  Note that each $(X,\mu_w)$ is a differentiability space in the sense
  of Cheeger \cite{cheeger99} with a unique chart
  $(X,\pi_{\infty,0})$. Then Cheeger's
  formulation of Rademacher's Theorem holds for a continuum of
  mutually singular measures.
\end{rem}
\bibliographystyle{alpha} 
\bibliography{sing_poinc_biblio}

\begin{thebibliography}{CK13b}

\bibitem[Che99]{cheeger99}
J.~Cheeger.
\newblock Differentiability of {L}ipschitz functions on metric measure spaces.
\newblock {\em Geom. Funct. Anal.}, 9(3):428--517, 1999.

\bibitem[CK13a]{cheeger_inverse_poinc}
J.~{Cheeger} and B.~{Kleiner}.
\newblock {Inverse limit spaces satisfying a Poincare inequality}.
\newblock {\em ArXiv e-prints}, December 2013.

\bibitem[CK13b]{cheeger_inverse_l1}
Jeff Cheeger and Bruce Kleiner.
\newblock Realization of metric spaces as inverse limits, and bilipschitz
  embedding in {$L_1$}.
\newblock {\em Geom. Funct. Anal.}, 23(1):96--133, 2013.

\bibitem[HK98]{heinonen98}
Juha Heinonen and Pekka Koskela.
\newblock Quasiconformal maps in metric spaces with controlled geometry.
\newblock {\em Acta Math.}, 181(1):1--61, 1998.

\bibitem[Laa00]{laakso_poinc}
T.~J. Laakso.
\newblock Ahlfors {$Q$}-regular spaces with arbitrary {$Q>1$} admitting weak
  {P}oincar\'e inequality.
\newblock {\em Geom. Funct. Anal.}, 10(1):111--123, 2000.

\bibitem[LP01]{lang_plaut}
Urs Lang and Conrad Plaut.
\newblock Bilipschitz embeddings of metric spaces into space forms.
\newblock {\em Geom. Dedicata}, 87(1-3):285--307, 2001.

\end{thebibliography}
\end{document}